\documentclass[reqno,11pt]{amsart}

\usepackage{amssymb}
\usepackage{amsgen}
\usepackage{amsmath}
\usepackage{amsthm}
\usepackage{mathrsfs}
\usepackage{cite}
\usepackage{amsfonts}
\usepackage{tikz-cd}
\usepackage{tikz}
\usetikzlibrary{shapes,decorations.pathreplacing}
\usetikzlibrary{calc}
\usetikzlibrary{backgrounds}
\usetikzlibrary{automata}
\usetikzlibrary{positioning}
\usetikzlibrary{decorations.markings}
\usetikzlibrary{arrows}
\usetikzlibrary{scopes}
\usetikzlibrary{intersections}
\usepackage{enumerate}

\newcommand{\simlift}[1]{#1^{\sharp}}

\newcommand{\sgn}{\mathop{\mathrm{sgn}}}
\newcommand{\gldim}{\mathop{\mathrm{gl.~dim}}}

\renewcommand{\to}{\longrightarrow}

\newcommand{\rad}[1]{\mathop{\mathrm{rad}}\nolimits(#1)}
\newcommand{\modu}[1]{#1\text{-}\mathrm{mod}}

\newcommand{\Res}{\mathop{\mathrm{Res}}\nolimits}
\newcommand{\Ind}{\mathop{\mathrm{Ind}}\nolimits}
\newcommand{\Coind}{\mathop{\mathrm{Coind}}\nolimits}

\newcommand{\JJ}{\mathrel{\mathscr J}} 
\newcommand{\RR}{\mathrel{\mathscr R}} 
\newcommand{\LL}{\mathrel{\mathscr L}} 

\newcommand{\til}[1]{\ensuremath{\widetilde {#1}}}

\newcommand{\Irr}{\mathop{\mathrm{Irr}}\nolimits}

\newcommand{\soc}[1]{\mathop{\mathrm{soc}}\nolimits(#1)}
\newcommand{\Ime}{T}

\newcommand{\Anne}{N}
\newcommand{\Aug}[1]{\mathop{\mathrm{Aug}}\nolimits(#1)}

\newcommand{\End}{\mathop{\mathrm{End}}\nolimits}

\newcommand{\Hom}{\mathop{\mathrm{Hom}}\nolimits}
\newcommand{\Ext}{\mathop{\mathrm{Ext}}\nolimits}

\newcommand{\chaptermarkb}[1]%
   {\markboth{{#1}}{}}
\newcommand{\sectionmarkb}[1]%
   {\markright{{#1}}{}}

   \newtheorem{theorem}{Theorem}[section]
\newtheorem{proposition}[theorem]{Proposition}

\newtheorem{lemma}[theorem]{Lemma}
{\theoremstyle{definition}
}
{\theoremstyle{remark}
}
\newtheorem{corollary}[theorem]{Corollary}
{\theoremstyle{remark}
}
{\theoremstyle{remark}
}

\newtheorem{prop}[theorem]{Proposition}

\newtheorem{cor}[theorem]{Corollary}
{\theoremstyle{remark}
}
{\theoremstyle{remark}
}
{\theoremstyle{remark}
}

\numberwithin{equation}{section}

\title[The global dimension of the full transformation monoid]{The global dimension of the full transformation monoid with an appendix by V.~Mazorchuk and B.~Steinberg}

\author{Benjamin Steinberg}
\address{%
    BS:
    Department of Mathematics\\
    City College of New York\\
    Convent Avenue at 138th Street\\
    New York, New York 10031\\
    USA}
\email{bsteinberg@ccny.cuny.edu}
\address{%
VM:
Department of Mathematics, Uppsala University, Box 480,
SE-75106,\\ Uppsala, SWEDEN}
\email{mazor\symbol{64}math.uu.se}

\thanks{This work was partially supported by a grant from the Simons Foundation(\#245268
to Benjamin Steinberg) and the Binational Science Foundation of Israel and the US (\#2012080 to Benjamin Steinberg), by a PSC-CUNY grant and by a CUNY Collaborative Incentive Research Grant. The research in the appendix was conducted by VM and BS during the Representation Theory program at the Institut Mittag-Leffler in May~2015.}
\date{April 16, 2015, Revised January 19, 2016}

\keywords{full transformation monoid, global dimension, monoid representation theory, quivers}
\subjclass[2010]{20M30, 16E10, 20M25, 16G99}

\begin{document}

\begin{abstract}
The representation theory of the symmetric group has been intensively studied for over 100 years and is one of the gems of modern mathematics.  The full transformation monoid $\mathfrak T_n$ (the monoid of all self-maps of an $n$-element set) is the monoid analogue of the symmetric group.  The investigation of its representation theory was begun by Hewitt and Zuckerman in 1957.  Its character table was computed by Putcha in 1996 and its representation type was determined in a series of papers by Ponizovski{\u\i}, Putcha and Ringel between 1987 and 2000.  From their work, one can deduce that the global dimension of $\mathbb C\mathfrak T_n$ is $n-1$ for $n=1,2,3,4$.  We prove in this paper that the global dimension is $n-1$ for all $n\geq 1$ and, moreover, we provide an explicit minimal projective resolution of the trivial module of length $n-1$.

In an appendix with V.~Mazorchuk we compute the indecomposable tilting modules of $\mathbb C\mathfrak T_n$ with respect to Putcha's quasi-hereditary structure and the Ringel dual (up to Morita equivalence).
\end{abstract}

\maketitle

\section{Introduction}
The character theory of the symmetric group (cf.~\cite{jameskerber,FultonHarris,MacDonaldIG,Sagan,okversh}) is an elegant piece of mathematics, featuring a beautiful blend of algebra and combinatorics, with applications to such diverse areas as probability~\cite{diaconisnotes,diaconisspectralanalysis} and mathematical physics.

The analogue in monoid theory of the symmetric group is the full transformation monoid $\mathfrak T_n$.  This is the monoid of all self-maps of an $n$-element set.  In 1957, Hewitt and Zuckerman initiated the study of the representation theory of $\mathfrak T_n$,  showing that the simple $\mathbb C\mathfrak T_n$-modules are parameterized by partitions of $r$ where $1\leq r\leq n$~\cite{HewZuck}.  However, very few of their results were specific to $\mathfrak T_n$, as witnessed by the fact their main theorem is a special case of a result for arbitrary finite monoids obtained at approximately the same time by Munn and Ponizovski{\u\i}~\cite{Poni,Munn1}.

In a \textit{tour de force} work~\cite{Putcharep5}, Putcha computed the character table of $\mathfrak T_n$ and gave an explicit description of all the simple $\mathbb C\mathfrak T_n$-modules except for one family.  It was the discovery of an explicit description of this second family via exterior powers that led to this paper.  Putcha, in fact, knows this explicit description (private communication), but only became aware of it after~\cite{Putcharep5,Putcharep3} were written.

Recall that a finite dimensional algebra $A$ has finite representation type\index{finite representation type} if there are only finitely many isomorphism classes of finite dimensional indecomposable $A$-modules.   Ponizovski{\u\i} proved that $\mathbb C\mathfrak T_n$ has finite representation type for $n\leq 3$ and conjectured that this was true for all $n$~\cite{ponireptype}.  Putcha disproved Ponizovski{\u\i}'s conjecture by showing that $\mathbb C\mathfrak T_n$ does not have finite representation type for $n\geq 5$~\cite{Putcharep3}. He did this by computing enough of the quiver of $\mathbb C\mathfrak T_n$ to see that $\mathbb C\mathfrak T_n/\mathrm{rad}^2(\mathbb C\mathfrak T_n)$ already does not have finite representation type.  Putcha also computed the quiver of $\mathbb C\mathfrak T_4$ and observed that  $\mathbb C\mathfrak T_4/\mathrm{rad}^2(\mathbb C\mathfrak T_4)$  does have finite representation type.  Ringel~\cite{ringel} computed a quiver presentation for $\mathbb C\mathfrak T_4$ and proved that it is of finite representation type.
It is an open question to compute the quiver of $\mathbb C\mathfrak T_n$ in full generality.

Ringel also proved that $\mathbb C\mathfrak T_4$ has global dimension $3$ by exhibiting the minimal projective resolutions of the simple modules. It is easy to check that the global dimension of $\mathbb C\mathfrak T_n$ is also $n-1$ for $n=1,2,3$ using the results of~\cite{Putcharep3}.  The main result of this paper is that the global dimension of $\mathbb C\mathfrak T_n$ is $n-1$ for all $n\geq 1$. As a byproduct of the proof, we also establish that the quiver of $\mathbb C\mathfrak T_n$ is acyclic.  In fact, these results hold \textit{mutatis mutandis} over any ground field of characteristic $0$ because $\mathbb Q$ is a splitting field for all symmetric groups and hence for all full transformation monoids.  In an appendix with Volodymyr Mazorchuk we describe the indecomposable tilting modules of $\mathbb C\mathfrak T_n$ with respect to Putcha's quasi-hereditary structure~\cite{Putcharep3}.  They turn out to be precisely the injective indecomposable modules, except the trivial module, and the simple projective one-dimensional module that restricts to the sign representation on permutations and is annihilated by all non-permutations.  We also compute the Ringel dual of $\mathbb CT_n$.

We prove the main theorem using homological techniques for working with the algebra of a von Neumann regular monoid developed by the author and Margolis in~\cite{rrbg}.  One could alternatively use the theory of quasi-hereditary algebras~\cite{quasihered,dlabringelqh}.  The key point is that outside of a single family the standard modules with respect to the natural quasi-hereditary structure found by Putcha~\cite{Putcharep3} are simple and standard modules have good homological properties.

The paper begins with a review of monoid representation theory and the character theory of the full transformation monoid.  The following section consists of two parts: the first part proves that the exterior powers of the natural $\mathbb C\mathfrak T_n$-module $\mathbb C^n$ are projective indecomposable modules with simple tops and simple radicals that are exterior powers of the augmentation submodule; the second part establishes a vanishing result for higher $\mathrm{Ext}$-functors between simple modules and proves the main result.  The appendix, with V.~Mazorchuk, uses the results of the previous sections to compute the indecomposable tilting modules and Ringel dual of $\mathbb C\mathfrak T_n$  with respect to its natural quasi-hereditary structure~\cite{Putcharep3}.

The reader is referred to~\cite{CP,Arbib,Eilenberg,higginsbook,qtheor} for  basic semigroup theory.  We use~\cite{AuslanderReiten,benson,assem} as our primary references for the theory of finite dimensional algebras. In this paper, all modules are assumed to be left modules and all monoids are assumed to act on the left of sets unless otherwise indicated. For a detailed discussion of the representation theory of finite monoids and its applications, the reader is referred to the author's forthcoming book~\cite{SteinbergBook}.

\section{The representation theory of finite monoids}
This section reviews the necessary background results for the paper.  Details can be found in~\cite{gmsrep,SteinbergBook}.

\subsection{Finite monoids}
Fix a finite monoid $M$. An \emph{ideal} $I$ of $M$ is a non-empty subset $I$ such that $MIM\subseteq I$.  Left and right ideals are defined analogously.   If $m\in M$, then $Mm$, $mM$ and $MmM$ are the \emph{principal} left, right and two-sided ideals generated by $m$, respectively. We put
\[I(m) = \{n\in M\mid m\notin MnM\}.\]
If $I(m)\neq \emptyset$, then it is an ideal.

The following equivalence relations are three of \emph{Green's relations}~\cite{Green}.  Set, for $m_1,m_2\in M$,
\begin{enumerate}[(i)]
\item $m_1\JJ m_2$ if and only if $Mm_1M=Mm_2M$;
\item $m_1\LL m_2$ if and only if $Mm_1=Mm_2$;
\item $m_1\RR m_2$ if and only if $m_1M=m_2M$.
\end{enumerate}
The $\mathscr J$-class of an element $m$ is denoted by $J_m$, and similarly the $\LL$-class and $\RR$-class of $m$ are denoted $L_m$ and $R_m$, respectively.

Let us describe Green's relations on $\mathfrak T_n$; this is a standard exercise in semigroup theory (cf.~\cite{CP,GM}).  The \emph{rank} of a mapping $f\in \mathfrak T_n$ is the cardinality of its image.  Let $f,g\in \mathfrak T_n$. Then $f\JJ g$ if and only if they have the same rank, $f\RR g$ if and only if they have the same image and $f\LL g$ if and only if they induce the same partition of the domain into fibers.

The set of idempotents of a monoid $M$ is denoted $E(M)$.  If $e\in E(M)$, then $eMe$ is a monoid with identity $e$.  Let $G_e$ be the group of units of $eMe$.  It is called the \emph{maximal subgroup} of $M$ at $e$. If $e,f\in E(M)$ with $MeM=MfM$, then $eMe\cong fMf$ and $G_e\cong G_f$ (cf.~\cite{TilsonXI}).

The group of units of $\mathfrak T_n$ is the symmetric group $\mathfrak S_n$.  If $e\in \mathfrak T_n$ has rank $r$, then it is well known and easy to show that $e\mathfrak T_ne\cong \mathfrak T_r$ and hence $G_e\cong \mathfrak S_r$.

A standard fact in the theory of finite monoids is the following~\cite{TilsonXI}.

\begin{proposition}\label{p:sing.elem}
Let $e\in E(M)$. Then $J_e\cap eMe=G_e$, $J_e\cap Me=L_e$ and $J_e\cap eM=R_e$.
\end{proposition}

The maximal subgroup $G_e$ at an idempotent $e$ acts freely on the right of $L_e$ and on the left of $R_e$, respectively, by multiplication and two elements of one of these sets are in the same $G_e$-orbit if and only if they are $\mathscr R$-equivalent, respectively, $\mathscr L$-equivalent; see~\cite{CP} or~\cite[Appendix A]{qtheor}.

A monoid $M$ is (von Neuman) \emph{regular} if, for all $m\in M$, there exists $m'\in M$ with $mm'm=m$.  For example, it is well known that the full transformation monoid $\mathfrak T_n$ is regular~\cite{CP,GM}, as is the monoid $M_n(\mathbb F)$ of all $n\times n$ matrices over a field $\mathbb F$~\cite{CP}.  A finite monoid is regular if and only if each $\mathscr J$-class contains an idempotent, cf.~\cite{Arbib} or~\cite[Appendix A]{qtheor}.

\subsection{Simple modules}
Let $\Bbbk$ be a field.  We continue to hold fixed a finite monoid $M$.   If $X\subseteq M$ we let $\Bbbk X$ denote the $\Bbbk$-linear span of $X$ in the monoid algebra $\Bbbk M$. If $I\subseteq M$ is a left (respectively, right or two-sided) ideal of $M$, then $\Bbbk I$ is a left (respectively, right or two-sided) ideal of $\Bbbk M$.

Let $S$ be a simple $\Bbbk M$-module.  We say that an idempotent $e\in E(M)$ is an \emph{apex}\index{apex} for $S$ if $eS\neq 0$ and $\Bbbk I(e)S=0$.  One has that $mS= 0$ if and only if $m\in I(e)$.  It follows that if $e,f$ are apexes of $S$, then $MeM=MfM$.

Fix an idempotent $e\in E(M)$ and put $A_e=\Bbbk M/\Bbbk I(e)$. Observe that $eA_ee\cong \Bbbk [eMe]/\Bbbk[eI(e)e]\cong \Bbbk G_e$ by Proposition~\ref{p:sing.elem}.  A simple $\Bbbk M$-module $S$ with apex $e$ is the same thing as a simple $A_e$-module $S$ with $eS\neq 0$. One can then apply the theory~\cite[Chapter 6]{Greenpoly} to classify these modules.  This was done in~\cite{gmsrep}.  See~\cite[Chapters~4,5]{SteinbergBook} for details.

Notice that as $\Bbbk$-vector spaces we have that $A_ee\cong \Bbbk L_e$ and $eA_e\cong \Bbbk R_e$~\cite{gmsrep}.  The corresponding left $\Bbbk M$-module structure on $\Bbbk L_e$ is defined by \[m\odot \ell =\begin{cases} m\ell, & \text{if}\ m\ell\in L_e\\ 0, & \text{else}\end{cases}\] for $m\in M$ and $\ell\in L_e$.  From now on we will omit the symbol ``$\odot$.''  The right $\Bbbk M$-module structure on $\Bbbk R_e$ is defined dually. In the semigroup theory literature, $\Bbbk L_e$ and $\Bbbk R_e$ are known as left and right \emph{Sch\"utzenberger representations}. Note that $\Bbbk L_e$ is a free right $\Bbbk G_e$-module and $\Bbbk R_e$ is a free left $\Bbbk G_e$-module because $G_e$ acts freely on the right of $L_e$ and on the left of $R_e$.  In fact, $\Bbbk L_e$ is a $\Bbbk M$-$\Bbbk G_e$-bimodule and dually $\Bbbk R_e$ is a $\Bbbk G_e$-$\Bbbk M$-bimodule.  Thus we can define functors
\begin{align*}
\Ind_{G_e}\colon& \modu{\Bbbk G_e}\to \modu{\Bbbk M}\\
\Coind_{G_e}\colon& \modu{\Bbbk G_e}\to \modu{\Bbbk M}\\
\Res_{G_e}\colon& \modu{\Bbbk M}\to \modu{\Bbbk G_e}\\
\Ime_e\colon& \modu{\Bbbk M}\to \modu{\Bbbk M}\\
\Anne_e\colon& \modu{\Bbbk M}\to \modu{\Bbbk M}
\end{align*}
by putting
\begin{align*}
\Ind_{G_e}(V) &= A_ee\otimes_{eA_ee}V = \Bbbk L_e\otimes_{\Bbbk G_e} V\\
\Coind_{G_e}(V) &= \Hom_{eA_ee}(eA_e,V)= \Hom_{\Bbbk G_e}(\Bbbk R_e,V)\\
\Res_{G_e}(V) &= eV =\Hom_{A_e}(A_ee,V)=eA_e\otimes_{A_e} V\\
\Ime_e(V) &= \Bbbk MeV\\
\Anne_e(V) & =\{v\in V\mid e\Bbbk Mv=0\}\}.
\end{align*}

One has that $\Res_{G_e}(\Ind_{G_e}(V))\cong V\cong \Res_{G_e}(\Coind_{G_e}(V))$ for any $\Bbbk G_e$-module, cf.~\cite{gmsrep}.  Note that the functors $\Ind_{G_e}$ and $\Coind_{G_e}$ are exact (cf.~\cite{gmsrep} or~\cite[Chapters~4,5]{SteinbergBook}) because $\Bbbk L_e$ and $\Bbbk R_e$ are free $\Bbbk G_e$-modules (on the appropriate sides). They also preserve indecomposability (see~\cite[Chapters~4,5]{SteinbergBook}).

If $V$ is a module over a finite dimensional algebra $A$, then $\rad{V}$ denotes the radical of $V$ and $\soc{V}$ the socle of $V$ (see~\cite{assem} for the definitions).

We now state the fundamental theorem of Clifford-Munn-Ponizovski{\u\i} theory~\cite[Chapter~5]{CP}, as formulated in~\cite{gmsrep}; see also~\cite[Chapter~5]{SteinbergBook}.

\begin{theorem}\label{t:clifford.munn.poni}
Let $M$ be a finite monoid and $\Bbbk$ a field.
\begin{enumerate}[(i)]
\item There is a bijection between isomorphism classes of simple $\Bbbk M$-modules with apex $e\in E(M)$ and isomorphism classes of  simple $\Bbbk G_e$-modules given on a simple $\Bbbk M$-module $S$ with apex $e$ and a simple $\Bbbk G_e$-module $V$ by
\begin{align*}
S&\longmapsto \Res_{G_e}(S)=eS\\
V&\longmapsto \simlift{V}=\Ind_{G_e}(V)/\Anne_e(\Ind_{G_e}(V))=\Ind_{G_e}(V)/\rad{\Ind_{G_e}(V)}\\ &\hspace{1.05cm}\cong \soc{\Coind_{G_e}(V)}=\Ime_e(\Coind_{G_e}(V)).
\end{align*}
\item Every simple $\Bbbk M$-module has an apex (unique up to $\mathscr J$-equivalence).
\item If $V$ is a simple $\Bbbk G_e$-module, then every composition factor of $\Ind_{G_e}(V)$ and $\Coind_{G_e}(V)$ has apex $f$ with $MeM\subseteq MfM$.  Moreover, $\simlift{V}$ is the unique composition factor of either of these two $\Bbbk M$-modules with apex $e$ and it appears in both these modules as a composition factor with multiplicity one.
\end{enumerate}
\end{theorem}

If we denote the set of isomorphism classes of simple $\Bbbk M$-modules by $\Irr_{\Bbbk}(M)$, then there is the following parametrization of the irreducible representations of $M$.

\begin{corollary}\label{c:param}
Let $e_1,\ldots, e_s$ be a complete set of idempotent representatives of the  $\mathscr J$-classes of $M$ containing idempotents.  Then there is a bijection between $\Irr_{\Bbbk}(M)$ and the disjoint union $\bigcup_{i=1}^s\Irr_{\Bbbk}(G_{e_i})$.
\end{corollary}

If $M$ is regular and $\Bbbk$ has characteristic $0$, then the modules $\Ind_{G_e}(V)$ and $\Coind_{G_e}(V)$ with $e\in E(M)$ and $V\in \Irr_{\Bbbk}(G_e)$ form the standard and costandard modules, respectively, for the natural structure quasi-hereditary algebra on $\Bbbk M$ found by Putcha~\cite{Putcharep3}. See~\cite{quasihered,dlabringelqh} for more on quasi-hereditary algebras.

\subsection{Homological aspects}
Let $A$ be a finite dimensional $\Bbbk$-algebra. The \emph{projective dimension} $\mathop{\mathrm{pd}} V$ of an $A$-module $V$ is the minimum length (possibly infinite) of a projective resolution of $V$. Each finite dimensional $A$-module $V$ has a unique minimal projective resolution (minimal in both length and in a certain categorical sense); see~\cite{AuslanderReiten,benson,assem}. Formally, a projective resolution $P_{\bullet}\to V$ is \emph{minimal} if each boundary map $d_n\colon P_n\to d_n(P_n)$ is a projective cover.
The following well-known proposition is stated in the context of group algebras in~\cite[Proposition~3.2.3]{cohomologyringbook}, but the proof there is valid for any finite dimensional algebra.

\begin{prop}\label{minresolution}
Let $A$ be a finite dimensional $\Bbbk$-algebra, $M$ a finite dimensional $A$-module and let $P_\bullet\to M$ be a projective resolution.  Then the following are equivalent.
\begin{enumerate}
\item $P_\bullet\to M$ is the minimal projective resolution of $M$.
\item $\Hom_A(P_q,S)\cong\Ext_A^q(M,S)$ for any $q\geq 0$ and simple $A$-module $S$.
\end{enumerate}
\end{prop}

The \emph{global dimension} of $A$ is
\[\gldim A=\sup \{\mathop{\mathrm{pd}} V\mid V\ \text{is an $A$-module}\}.\]

It is convenient to reformulate the definition in terms of the $\Ext$-functors.  The $\Ext$-functors measure the failure of the $\Hom$-functors to be exact.  Details about their properties can be found in~\cite{benson,assem} or any text on homological algebra.
A classical fact is that
\[\gldim A=\sup\{n\in \mathbb N\mid \Ext^n_A(S,S')\neq 0,\ \text{for some simple modules}\ S,S'\}\]
where the supremum could be infinite~\cite{AuslanderReiten,benson,assem}.
This reformulation relies on the following well-known lemma, which is proved by induction on the number of composition factors using the long exact sequence associated to the $\Ext$-functors.

\begin{lemma}\label{l:comp.fact}
Let $V,W$ be finite dimensional $A$-modules.  Then one has that $\Ext^n_A(V,W)=0$ if either of the following two conditions hold.
\begin{enumerate}[(i)]
\item $\Ext^n_A(V,S)=0$ for each composition factor $S$ of $W$.
\item $\Ext^n_A(S',W)=0$ for each composition factor $S'$ of $V$.
\end{enumerate}
\end{lemma}

 The following result is~\cite[Lemma~3.3]{rrbg}.

\begin{lemma}\label{l:ideal.reg.mon}
Let $M$ be a finite regular monoid and $\Bbbk$ a field.  Let $I$ be an ideal of $M$.  Then the isomorphism \[\Ext^n_{\Bbbk M}(V,W)\cong \Ext^n_{\Bbbk M/\Bbbk I}(V,W)\] holds for any $\Bbbk M/\Bbbk I$-modules $V,W$ and all $n\geq 0$.
\end{lemma}

The author and Margolis proved in~\cite[Lemma~3.5]{rrbg} the following lemma in the same vein as the Eckmann-Shapiro lemma from group cohomology.

\begin{lemma}\label{l:eck.shap}
Let $M$ be a finite regular monoid and $\Bbbk$ a field.  Let $e\in E(M)$ and  $I=MeM\setminus J_e$.  Then, for any $\Bbbk G_e$-module $V$ and $\Bbbk M/\Bbbk I$-module $W$, one has natural isomorphisms
\begin{align*}
\Ext^n_{\Bbbk M}(\Ind_{G_e}(V),W)&\cong \Ext^n_{\Bbbk G_e}(V,\Res_{G_e}(W))\\
\Ext^n_{\Bbbk M}(W,\Coind_{G_e}(V))&\cong \Ext^n_{\Bbbk G_e}(\Res_{G_e}(W),V)
\end{align*}
for all $n\geq 0$.
\end{lemma}

Since $\Bbbk G_e$ is semisimple whenever $\Bbbk$ is of characteristic zero, we obtain the following corollary.

\begin{corollary}\label{c:kill.above}
Let $M$ be a finite regular monoid and $\Bbbk$ a field of characteristic $0$.  Let $e\in E(M)$ and let $I=MeM\setminus J_e$. Then, for any $\Bbbk G_e$-module $V$ and $\Bbbk M/\Bbbk I$-module $W$, one has
\begin{align*}
\Ext^n_{\Bbbk M}(\Ind_{G_e}(V),W)=0\\
\Ext^n_{\Bbbk M}(W,\Coind_{G_e}(V))=0
\end{align*}
for all $n\geq 1$.
\end{corollary}

Nico~\cite{Nico1,Nico2}  proved that the global dimension of a regular monoid over a field of characteristic zero is always finite.
\begin{theorem}[Nico]
Let $M$ be a finite regular monoid and let $\Bbbk$ be a field of characteristic $0$.  Then $\gldim \Bbbk M$ is bounded by $2(m-1)$ where $m$ is the length of the longest chain of non-zero principal ideals of $M$.
\end{theorem}

\subsection{The character theory of the full transformation monoid}

The character theory of $\mathfrak T_n$ has a very long history, beginning with the work of Hewitt and Zuckerman~\cite{HewZuck}. A complete computation of the character table of $\mathfrak T_n$ was finally achieved by Putcha nearly 40 years later~\cite[Theorem~2.1]{Putcharep5}.  To formulate his result, first let $e_r\in \mathfrak T_n$ be the idempotent given by
\begin{equation}\label{eq:some.idems}
e_r(i) = \begin{cases}i, & \text{if}\ i\leq r\\ 1, & \text{if}\ i>r\end{cases}
\end{equation}
and note that $e_1,\ldots, e_n$ form a complete set of idempotent representatives of the $\mathscr J$-classes  of $\mathfrak T_n$ and $e_r\mathfrak T_ne_r\cong \mathfrak T_r$, whence $G_{e_r}\cong \mathfrak S_r$. The isomorphism takes $f\in e_r\mathfrak T_ne_r$ to $f|_{[r]}$ where $[r]=\{1,\ldots,r\}$.

The reader is referred to~\cite{jameskerber} for the representation theory of the symmetric group.  If $\lambda$ is a partition of $n$, then $S(\lambda)$ will denote the corresponding simple module (Specht module).  We put $L(\lambda)=\simlift{S(\lambda)}$ to avoid cumbersome notation.
Let us use $1^k$ as short hand for a sequence of $k$ ones occurring in a partition. With this notation, $S((1^r))$ is the sign representation of $\mathfrak S_r$. On the other hand, $S((r))$ is the trivial representation of $\mathfrak S_r$.    Thus we consider $S((1)))$ to be both the trivial and the sign representation of  $\mathfrak S_1$.

\begin{theorem}[Putcha]\label{t:full.trans.char}
Fix $n\geq 1$ and let $S(\lambda)\in \Irr_{\mathbb C}(\mathfrak S_r)$ for $1\leq r\leq n$.
\begin{enumerate}[(i)]
\item If $\lambda\neq (1^r)$, then $\Ind_{\mathfrak S_r}(S(\lambda))$ is simple (and hence equal to $L(\lambda)$). Moreover, its restriction to $\mathbb C\mathfrak S_n$ is isomorphic to the induced module $\mathbb C\mathfrak S_n\otimes_{\mathbb C[\mathfrak S_r\times S_{n-r}]}(S(\lambda)\otimes S((n-r)))$.
\item If $\lambda=(1^r)$, then  the restriction of $L((1^r))$  to $\mathbb C\mathfrak S_n$ is the simple module $S((n-r+1,1^{r-1}))$ of dimension $\binom{n-1}{r-1}$.
\end{enumerate}
\end{theorem}

Let us remark that Putcha~\cite{Putcharep5} uses very different notation than ours.  If $\theta$ is an irreducible representation of $\mathfrak S_r$ afforded by a simple module $V$, then Putcha uses $\theta^+$ to denote the representation afforded by $\Ind_{\mathfrak S_r}(V)$ and $\theta^-$ to denote the representation afforded by $\Coind_{\mathfrak S_r}(V)$.  He uses $\til \theta$ for the representation afforded by $\simlift{V}$.  A proof of Theorem~\ref{t:full.trans.char} can be found in~\cite[Chapter~5]{SteinbergBook}.

We proceed to give a new proof of Theorem~\ref{t:full.trans.char}(ii), which is simpler than Putcha's, by exhibiting the simple module.   Putcha has informed the author that he is aware of this construction.    The key observation is that $S((n-r+1,1^{r-1}))$ is an exterior power of the representation $S((n-1,1))$ of $\mathfrak S_n$.

Note that $\mathbb C^n$ is a $\mathbb C\mathfrak T_n$-module by defining $fv_i = v_{f(i)}$, for $f\in \mathfrak T_n$, where $v_1,\ldots, v_n$ denotes the standard basis for $\mathbb C^n$.  We call this the \emph{natural module}.  Moreover, the \emph{augmentation}
\[\Aug{\mathbb C^n}=\{(x_1,\ldots,x_n)\in \mathbb C^n\mid x_1+\cdots+x_n=0\}\] is a $\mathbb C\mathfrak T_n$-submodule and $\mathbb C^n/\Aug{\mathbb C^n}$ is the trivial $\mathbb C\mathfrak T_n$-module.

\begin{theorem}\label{t:alt.rep.ft}
Let $n\geq 1$ and let $V=\Aug{\mathbb C^n}$.  Then the exterior power $\Lambda^{r-1}(V)$,  for $1\leq r\leq n$, is a simple $\mathbb C\mathfrak T_n$-module with apex $e_{r}$ and with $e_r\Lambda^{r-1}(V)\cong S((1^r))$ the sign representation of $G_{e_r}\cong \mathfrak S_r$.  In other words, $L((1^r))=\Lambda^{r-1}(V)$.
\end{theorem}
\begin{proof}
If $1\leq m\leq n$, then  clearly $e_m\mathbb C^n\cong \mathbb C^m$ and $e_m\Aug{\mathbb C^n}\cong \Aug{\mathbb C^m}$ under the identification of $e_m\mathbb C \mathfrak T_ne_m=\mathbb C[e_m\mathfrak T_ne_m]$ with $\mathbb C\mathfrak T_m$ induced by restricting an element of $e_m\mathfrak T_ne_m$ to $[m]$.   In particular, as $e_m\Lambda^{r-1}(V)=\Lambda^{r-1}(e_mV)$ and $\dim e_mV=\dim \Aug{\mathbb C^m}=m-1$, we conclude that $e_m\Lambda^{r-1}(V)=0$ if $m<r$.

Let us now assume that $m\geq r$.  By~\cite[Proposition~3.12]{FultonHarris}, the exterior power $\Lambda^{r-1}(e_mV)\cong \Lambda^{r-1}(\Aug{\mathbb C^m})$ is a simple $\mathbb C\mathfrak S_m$-module of degree $\binom{m-1}{r-1}$ and by~\cite[Exercise~4.6]{FultonHarris} it is, in fact, the Specht module $S((m-r+1,1^{r-1}))$.  In particular, we have that $e_r\Lambda^{r-1}(V)=S((1^r))$.  It follows that $W=\Lambda^{r-1}(V)$ is a simple $\mathbb C\mathfrak T_n$-module with apex $e_r$ and that $W\cong L((1^r))$.
\end{proof}

\section{The global dimension of $\mathbb C\mathfrak T_n$}
In this section, we prove that the global dimension of $\mathbb C\mathfrak T_n$ is $n-1$.  Our first goal is to provide the minimal projective resolutions of the exterior powers of $\Aug{\mathbb C^n}$.  We retain the notation of the previous section.

\subsection{Minimal projective resolutions of the exterior powers}
We continue to denote the natural $\mathbb C\mathfrak T_n$-module by $\mathbb C^n$.  It turns out that each exterior power of $\mathbb C^n$ is a projective indecomposable module.

\begin{theorem}\label{t:ext.is.proj}
For each $r$ with $1\leq r\leq n$, the $\mathbb C\mathfrak T_n$-module $\Lambda^r(\mathbb C^n)$ is a projective indecomposable module.
\end{theorem}
\begin{proof}
Let $e_r$ be the idempotent from \eqref{eq:some.idems} and $W=\mathbb C^n$. Let \[\eta=\frac{1}{r!}\sum_{g\in G_{e_r}}\sgn(g|_{[r]})g.\] We will prove that $\eta$ is a primitive idempotent of $\mathbb C\mathfrak T_n$ and that $\mathbb C\mathfrak T_n\eta\cong \Lambda^r(W)$.

Again denote by $v_1,\ldots, v_n$ the standard basis for $W$.  Observe that $\mathfrak T_ne_r$ consists of all the mappings $f\in \mathfrak T_n$ with $f(x)=f(1)$ for $x>r$.  We can thus define an isomorphism of $\mathbb C\mathfrak T_n$-modules $\rho\colon \mathbb C\mathfrak T_ne_r\to W^{\otimes r}$ by \[\rho(f) = v_{f(1)}\otimes \cdots \otimes v_{f(r)}\] for $f\in \mathfrak T_ne_r$.  We can identify $G_{e_r}$ with the symmetric group $\mathfrak S_r$ via the isomorphism $g\mapsto g|_{[r]}$.  Under this identification, we see that if $g\in G_{e_r}$ and $f\in \mathfrak T_ne_r$, then  \[\rho(fg) = v_{f(g(1))}\otimes \cdots\otimes v_{f(g(r))} = \rho(f)g|_{[r]}\] where $\mathfrak S_r$ acts on the right of $W^{\otimes r}$ by permuting the tensor factors.

We have that $\eta$ is a primitive idempotent of $\mathbb CG_{e_r}$ and $\mathbb CG_{e_r}\eta$ affords the sign representation of $G_{e_r}\cong \mathfrak S_r$.
By definition, $\Lambda^r(W) = W^{\otimes r}\otimes_{\mathbb C\mathfrak S_r} S((1^r))$ and hence \[\Lambda^r(\mathbb C^n)\cong \mathbb C\mathfrak T_ne_r\otimes_{\mathbb CG_{e_r}} S((1^r))=\mathbb C\mathfrak T_ne_r\otimes_{\mathbb CG_{e_r}} \mathbb CG_{e_r}\eta\cong \mathbb C\mathfrak T_n\eta.\]  Since $\eta\in \mathbb C\mathfrak T_n$ is an idempotent, we deduce that $\Lambda^r(\mathbb C^n)$ is a projective $\mathbb C\mathfrak T_n$-module.

It remains to prove that $\eta$ is primitive or, equivalently, that $\eta\mathbb C\mathfrak T_n\eta$ is a local ring.  In fact, we show that it is isomorphic to $\mathbb C$.  Indeed, as a $\mathbb CG_{e_r}$-module we have that
\[e_r\mathbb C\mathfrak T_n\eta\cong e_r\Lambda^r(W)=\Lambda^r(e_rW)\cong S((1^r))\] and hence there is a vector space isomorphism $\eta\mathbb C\mathfrak T_n\eta\cong \eta S((1^r))\cong \mathbb C$, as $\eta$ is the primitive idempotent of $\mathbb CG_{e_r}$ corresponding to $S((1^r))$.  Since $\dim \eta\mathbb C\mathfrak T_n\eta=1$, we conclude that $\eta\mathbb C\mathfrak T_n\eta\cong\mathbb C$ as a $\mathbb C$-algebra.  This completes the proof that $\eta$ is primitive and hence $\Lambda^r(W)$ is a projective indecomposable $\mathbb C\mathfrak T_n$-module.
\end{proof}

A crucial observation for understanding the representation theory of $\mathfrak T_n$ is the following short exact sequence for the projective indecomposable module $\Lambda^r(\mathbb C^n)$, which can be viewed as a categorification of Pascal's identity for binomial coefficients.

\begin{theorem}\label{t:ext.power.seq}
Let $1\leq r\leq n$ and let $V=\mathrm{Aug}(\mathbb C^n)$.  Then there is a short exact sequence of $\mathbb C\mathfrak T_n$-modules
\[\begin{tikzcd}0\ar{r} & \Lambda^{r}(V)\ar{r} & \Lambda^r(\mathbb C^n)\ar{r} & \Lambda^{r-1}(V)\ar{r} &0\end{tikzcd}\]
which does not split if $1\leq r<n$.  Note that $\Lambda^n(\mathbb C^n)\cong \Lambda^{n-1}(V)$ is the projective simple module $L((1^n))$.
\end{theorem}
\begin{proof}
Clearly, $\Lambda^r(V)$ is a submodule $\Lambda^r(\mathbb C^n)$.  Let $v_1,\ldots, v_n$ be the standard basis for $\mathbb C^n$ and  put $w_i = v_i-v_n$  for $i=1,\ldots, n-1$.  Then $w_1,\ldots, w_{n-1}$ is a basis for $V$.  We claim that $\Lambda^r(\mathbb C^n)/\Lambda^r(V)\cong \Lambda^{r-1}(V)$.

Put $w_n=v_1+\cdots+v_n$.  Then $w_1,\ldots, w_n$ is a basis for $\mathbb C^n$.  Define $\rho\colon \Lambda^r(\mathbb C^n)\to \Lambda^{r-1}(V)$ on the basis of $r$-fold wedge products of $w_1,\ldots, w_n$ by
\[\rho(w_{i_1}\wedge\cdots \wedge w_{i_r})=\begin{cases} w_{i_1}\wedge \cdots \wedge w_{i_{r-1}}, & \text{if}\ i_r=n\\ 0, & \text{else}\end{cases}\]
for $1\leq i_1<\cdots<i_r\leq n$.  This is clearly a surjective linear map with kernel $\Lambda^r(V)$.  Let us check that it is a $\mathbb C\mathfrak T_n$-module homomorphism.
If $g\in \mathfrak T_n$, then since $\mathbb C^n/V$ is the trivial module, it follows that $gw_n +V= w_n + V$ and so $gw_n=w_n+v_g$ with $v_g\in V$.  Therefore, if
$1\leq i_1<\cdots<i_{r-1}\leq n-1$, then
\begin{align*}
g(w_{i_1}\wedge\cdots \wedge w_{i_{r-1}}\wedge w_n) &= gw_{i_1}\wedge\cdots \wedge gw_{i_{r-1}}\wedge w_n+{}\\ &\qquad\quad gw_{i_1}\wedge\cdots \wedge gw_{i_{r-1}} \wedge v_g\\ &\in gw_{i_1}\wedge\cdots \wedge gw_{i_{r-1}}\wedge w_n+\Lambda^r(V).
\end{align*}
We conclude that \[g\rho(w_{i_1}\wedge\cdots \wedge w_{i_{r-1}}\wedge w_n)=gw_{i_1}\wedge\cdots \wedge gw_{i_{r-1}}=\rho(g(w_{i_1}\wedge\cdots \wedge w_{i_{r-1}}\wedge w_n)).\]
This completes the proof that there is such an exact sequence.  If $1\leq r<n$, then it cannot split because $\Lambda^r(\mathbb C^n)$ is indecomposable by Theorem~\ref{t:ext.is.proj}.
\end{proof}

Let us deduce the following important corollary.

\begin{cor}\label{c:proj.indec}
For $1\leq r\leq n$, the exterior power $\Lambda^r(\mathbb C^n)$ is a projective indecomposable module with $\rad{\Lambda^r(\mathbb C^n)}=\Lambda^r(\Aug{\mathbb C^n})$ and simple top $\Lambda^{r-1}(\Aug{\mathbb C^n})$.
\end{cor}
\begin{proof}
Theorem~\ref{t:ext.is.proj} yields that $P=\Lambda^r(\mathbb C^n)$ is a projective indecomposable module and hence has a simple top.  As the modules $\Lambda^r(\Aug{\mathbb C^n})$ and $\Lambda^{r-1}(\Aug{\mathbb C^n})$ are simple by Theorem~\ref{t:alt.rep.ft}, Theorem~\ref{t:ext.power.seq} yields that $\rad{P}=\Lambda^r(\Aug{\mathbb C^n})$ and $P/\rad{P}\cong \Lambda^{r-1}(\Aug{\mathbb C^n})$.
\end{proof}

Theorem~\ref{t:ext.power.seq} allows us to construct the minimal projective resolution of $\Lambda^{r-1}(\Aug{\mathbb C^n})$ for $1\leq r\leq n$.

\begin{cor}\label{c:minimal.res}
Let $v_1,\ldots, v_n$ be the standard basis for $\mathbb C^n$, let $w_i=v_i-v_n$ for $1\leq i\leq n-1$ and let $w_n=v_1+\cdots+v_n$.  Let $V=\Aug{\mathbb C^n}$. Then, for $1\leq r\leq n$, the minimal projective resolution of the simple module $\Lambda^{r-1}(V)$ is
\[\begin{tikzcd}0\ar{r} & P_{n-r}\ar{r}{d_{n-r}}& \cdots\ar{r}{d_{1}} &P_0\ar{r}{d_0} & \Lambda^{r-1}(V)\ar{r} & 0\end{tikzcd}\] where $P_q=\Lambda^{q+r}(\mathbb C^n)$ and
\[d_{q}(w_{i_1}\wedge\cdots \wedge w_{i_{q+r}})=\begin{cases} w_{i_1}\wedge \cdots \wedge w_{i_{q+r-1}}, & \text{if}\ i_{q+r}=n\\ 0, & \text{else.}\end{cases}\]

Therefore, $\mathop{\mathrm{pd}} \Lambda^{r-1}(V)=n-r$ for $1\leq r\leq n$.  In particular, the projective dimension of the trivial $\mathbb C\mathfrak T_n$-module is $n-1$.
\end{cor}
\begin{proof}
The exactness of the resolution follows from repeated application of Theorem~\ref{t:ext.power.seq} (and its proof).  Corollary~\ref{c:proj.indec} implies that each mapping \[d_q\colon \Lambda^{q+r}(\mathbb C^n)\to d_q(\Lambda^{q+r}(\mathbb C^n))=\Lambda^{q+r-1}(V)\] is a projective cover and so the resolution is a minimal projective resolution.
The final statement holds because $\Lambda^0(V)$ is the trivial $\mathbb C\mathfrak T_n$-module.
\end{proof}

It follows from Corollary~\ref{c:minimal.res} that the global dimension of $\mathbb C\mathfrak T_n$ is at least $n-1$. The next subsection will show that this lower bound is tight. In fact, Corollary~\ref{c:minimal.res} yields that the cohomological dimension of $\mathfrak T_n$ over $\mathbb C$ is $n-1$ (the \emph{cohomological dimension} of a monoid $M$ over a base ring $R$ is the projective dimension of the trivial $RM$-module).

\subsection{A computation of the global dimension}

We first compute $\Ext$ from an exterior power of the augmentation submodule of $\mathbb C^n$.

\begin{proposition}\label{p:ext.tn.lower}
For $1\leq k,r\leq n$ and $\lambda$ a partition of $k$, one has that
\[\Ext^m_{\mathbb C\mathfrak T_n}(L((1^r)),L(\lambda))\cong \begin{cases}\mathbb C, & \text{if}\ \lambda=(1^{r+m})\\ 0, & \text{else.}\end{cases}\]
\end{proposition}
\begin{proof}
Let $V=\Aug{\mathbb C^n}$.  Recall that $L((1^r))\cong \Lambda^{r-1}(V)$ by Theorem~\ref{t:alt.rep.ft}. Using the minimal projective resolution for $\Lambda^{r-1}(V)$ from Corollary~\ref{c:minimal.res} and Proposition~\ref{minresolution}, we deduce that
\begin{align*}
\Ext^m_{\mathbb C\mathfrak T_n}(\Lambda^{r-1}(V),L(\lambda))&\cong \Hom_{\mathbb C\mathfrak T_n}(\Lambda^{r+m}(\mathbb C^n),L(\lambda))\\ &\cong \Hom_{\mathbb C\mathfrak T_n}(\Lambda^{r+m-1}(V),L(\lambda))
\end{align*}
where the last isomorphism uses that $\Lambda^{r+m}(\mathbb C^n)$ has simple top $\Lambda^{r+m-1}(V)$ by Corollary~\ref{c:proj.indec}.  Recalling that $\Lambda^{r+m-1}(V)\cong L((1^{r+m}))$ by Theorem~\ref{t:alt.rep.ft}, the result follows.
\end{proof}

We next prove a vanishing result when the first variable is not one of the exterior powers of the augmentation submodule of $\mathbb C^n$ (and hence is an induced module).

\begin{proposition}\label{p:other.mods.tn}
Let $n\geq 1$ and $1\leq k,r\leq n$.  Let $S(\lambda)$ be a simple $\mathbb C\mathfrak S_r$-module with $\lambda\neq (1^r)$ and let $S(\mu)$ be a simple $\mathbb C\mathfrak S_k$-module.  Then $\Ext^m_{\mathbb C\mathfrak T_n}(L(\lambda),L(\mu))=0$  unless $0\leq m\leq r-k\leq n-1$.
\end{proposition}
\begin{proof}
Suppose first that $k\geq r$.  Let $I_{r-1}\subseteq \mathfrak T_n$ be the ideal of mappings of rank at most $r-1$, where we take $I_0=\emptyset$, and put $A=\mathbb C\mathfrak T_n/\mathbb CI_{r-1}$.   Since  $L(\lambda)=\Ind_{G_{e_r}}(S(\lambda))$ by Theorem~\ref{t:full.trans.char} and $L(\mu)$ is an $A$-module, we conclude that $\Ext^m_{\mathbb C\mathfrak T_n}(L(\lambda),L(\mu))=0$ for all $m\geq 1$ by Corollary~\ref{c:kill.above}.  Since there are no homomorphisms from $L(\lambda)$ to $L(\mu)$ if $k>r$, this completes the first case.

Next suppose that $k\leq r$.  We proceed by induction on $r-k$ where the base case $r=k$ has already been handled.
Assume that $k<r$ and that the result is true for $k'$ with $k<k'\leq n$. Let $A=\mathbb C\mathfrak T_n/\mathbb CI_{k-1}$ and consider the exact sequence of $A$-modules
\[\begin{tikzcd}0\ar{r} & L(\mu)\ar{r}& \Coind_{G_{e_k}}(S(\mu))\ar{r} & \Coind_{G_{e_k}}(S(\mu))/L(\mu)\ar{r} &0 .\end{tikzcd}\] Since $L(\lambda)$ is an $A$-module, Corollary~\ref{c:kill.above} and the long exact sequence for $\mathrm{Ext}$ imply that, for $m>r-k\geq 1$,
\begin{equation*}
\Ext_{\mathbb C\mathfrak T_n}^{m}(L(\lambda),L(\mu))\cong  \Ext_{\mathbb C\mathfrak T_n}^{m-1}(L(\lambda), \Coind_{G_{e_k}}(S(\mu))/L(\mu))=0
\end{equation*}
 where the last equality uses that each composition factor of the module $\Coind_{G_{e_k}}(S(\mu))/L(\mu)$ has apex $e_{k'}$ with $k'>k$ by Theorem~\ref{t:clifford.munn.poni}, induction and Lemma~\ref{l:comp.fact}.  This completes the proof.
\end{proof}

We are now prepared to prove the main result of the paper.

\begin{theorem}\label{t:gl.dim.tn}
The global dimension of  $\mathbb C\mathfrak T_n$ is $n-1$ for all $n\geq 1$.
\end{theorem}
\begin{proof}
By Proposition~\ref{p:ext.tn.lower}, we have that  $\Ext^{n-1}_{\mathbb C\mathfrak T_n}(L((1)),L((1^n)))\cong \mathbb C$.  On the other hand, Proposition~\ref{p:ext.tn.lower} and Proposition~\ref{p:other.mods.tn} yield that \[\Ext^m_{\mathbb C\mathfrak T_n}(L(\lambda),L(\mu))=0\] for all simple modules $L(\lambda),L(\mu)$ and $m\geq n$.  This completes the proof that $\gldim \mathbb C\mathfrak T_n=n-1$.
\end{proof}

We remark that since $L((1))$ is a simple injective module (because it is isomorphic to $\Hom_{\mathbb C}(e_1\mathbb C\mathfrak T_n,\mathbb C)$ where $e_1$ is the constant mapping to $1$) it follows that if $\lambda\neq (1)$, then $\Ext^{n-1}_{\mathbb C\mathfrak T_n}(L(\lambda),L(\mu))=0$ for all partitions $\mu$ and hence no simple module other than the trivial module has projective dimension $n-1$.

In fact, since $\mathbb Q$ is a splitting field for all symmetric groups, and hence for all full transformation monoids, the above argument works \textit{mutatis mutandis} to prove that $\Bbbk \mathfrak T_n$ has global dimension $n-1$ for all fields $\Bbbk$ of characteristic $0$.

Recall that if $A$ is a finite dimensional algebra over an algebraically closed field, then the \emph{quiver} of $A$ is the directed graph with vertices the isomorphism classes of simple $A$-modules and edges as follows. If $S_1$ and $S_2$ are simple $A$-modules, the number of directed edges from the isomorphism class of $S_1$ to the isomorphism class of $S_2$ is $\dim \Ext^1_A(S_1,S_2)$.  See~\cite{AuslanderReiten,benson,assem} for details.

\begin{corollary}
The quiver of $\mathbb C\mathfrak T_n$ is acyclic for all $n\geq 1$.
\end{corollary}
\begin{proof}
Proposition~\ref{p:ext.tn.lower} implies that the only arrow exiting $L(1^r))$ is the arrow $L((1^r))\to L((1^{r+1}))$ for $1\leq r\leq n-1$ and that $L((1^n))$ is a sink.  Proposition~\ref{p:other.mods.tn} implies that all other arrows go from a module with apex $e_k$ to a module with apex $e_{j}$ with $j<k$.  It follows immediately that the quiver of $\mathbb C\mathfrak T_n$ is acyclic.
\end{proof}

It is an open question to compute the quiver of $\mathbb C\mathfrak T_n$ for $n\geq 5$.  The quiver for $1\leq n\leq 4$ can be found in~\cite{Putcharep3}.

\appendix
\renewcommand{\thesection}{\Alph{section}}
\section{Tilting modules and the Ringel dual for the full transformation monoid \\ by
Volodymyr Mazorchuk and Benjamin Steinberg}
\subsection{Quasi-hereditary algebras and Ringel dual}\label{s2}

Let $\Bbbk$ be a field, $A$ a finite dimensional $\Bbbk$-algebra and $\{L(\lambda)\mid \lambda\in\Lambda\}$
a fixed set of representatives of isomorphism classes of simple $A$-modules, where $\Lambda$
is a fixed finite index set. For $\lambda\in\Lambda$, we fix an indecomposable projective cover
$P(\lambda)$ of $L(\lambda)$ and an indecomposable injective envelope $I(\lambda)$ of $L(\lambda)$.

Fix a partial order $<$ on $\Lambda$ and let $\leq$ be the union of $<$ with the equality relation.
Let $\Delta(\lambda)$, where $\lambda\in\Lambda$, denote the quotient of $P(\lambda)$ by the submodule
generated by the images of all possible homomorphisms $P(\mu)\to P(\lambda)$, where $\mu\not\leq\lambda$.
The modules $\Delta(\lambda)$, where $\lambda\in\Lambda$, are called \emph{standard modules}.
Let $\nabla(\lambda)$, where $\lambda\in\Lambda$, denote the submodule of $I(\lambda)$ defined as
the intersection of the kernels of all possible homomorphisms $I(\lambda)\to I(\mu)$, where $\mu\not\leq\lambda$.
The modules $\nabla(\lambda)$, where $\lambda\in\Lambda$, are called \emph{costandard modules}.

The pair $(A,<)$ is called a \emph{quasi-hereditary algebra}, see~\cite{Scott,quasihered,dlabringelqh}, provided that
\begin{itemize}
\item the endomorphism algebra of each $\Delta(\lambda)$ is a division algebra;
\item each $P(\lambda)$ has a \emph{standard filtration}, that is, a filtration whose
subquotients are isomorphic to standard modules.
\end{itemize}
Equivalently, $(A,<)$ is quasi-hereditary  provided that
\begin{itemize}
\item the endomorphism algebra of each $\nabla(\lambda)$ is a division algebra;
\item each $I(\lambda)$ has a \emph{costandard filtration}, that is, a filtration
whose subquotients are isomorphic to costandard modules.
\end{itemize}
We refer the reader to the appendices in~\cite{drozd,Donkin} and to~\cite{KK} for more details.

Following~\cite{Ringeldual}, an $A$-module $T$ is called a \emph{tilting module} provided that
it has both a standard filtration and a costandard filtration. If $(A,<)$ is quasi-hereditary, then,
for any $\lambda\in\Lambda$, there is a unique (up to isomorphism) indecomposable tilting module
$T(\lambda)$ which contains $\Delta(\lambda)$ and such that the cokernel of the corresponding
inclusion $\Delta(\lambda)\hookrightarrow T(\lambda)$ has a standard filtration. Moreover,
every tilting module is isomorphic to a direct sum of (copies of) these $T(\lambda)$. The module
\begin{displaymath}
T=\bigoplus_{\lambda\in\Lambda}T(\lambda)
\end{displaymath}
is called the \emph{characteristic tilting module} for $A$. The algebra
$\mathrm{End}_{A}(T)^{\mathrm{op}}$ is quasi-hereditary and is called the \emph{Ringel dual} of $A$.
We refer to~\cite{Ringeldual,KK} for more details.

\subsection{Representation theory of the full transformation monoid}\label{s3}

\subsubsection{The full transformation monoid}\label{s3.1}

We fix a field $\Bbbk$ of characteristic $0$ and $n\geq 1$. We will continue to use the notation of the previous sections.
Simple $\Bbbk \mathfrak T_n$-modules are parameterized by partitions
$\lambda=(\lambda_1,\ldots, \lambda_s)$ of $r$ where $1\leq r\leq n$. We denote by $\Lambda$ the set of all such partitions.
We write $|\lambda|=\lambda_1+\cdots+\lambda_s$.  A special case of Putcha's results~\cite{Putcharep3} shows that $\Bbbk \mathfrak T_n$ is  quasi-hereditary with respect to the partial order that puts $\lambda <\rho$ if $|\lambda|>|\rho|$ for $\lambda,\rho\in \Lambda$. For each $\lambda\in\Lambda$,
we fix a corresponding simple $\mathfrak T_n$-module $L(\lambda)$. For each $\lambda\in\Lambda$, we denote by
$S(\lambda)$ the Specht $\Bbbk \mathfrak S_{|\lambda|}$-module  corresponding to $\lambda$.

\subsubsection{Standard and costandard $\Bbbk\mathfrak T_n$-modules}\label{s3.2}

The following descriptions of the standard and costandard $\Bbbk \mathfrak T_n$-modules
(with respect to this quasi-he\-red\-i\-tar\-y structure) are well known, see for example~\cite{Putcharep3}.  We continue to use the idempotents $e_1,\ldots, e_n$ from earlier.  The standard modules then turn out to be the induced modules and the costandard modules are the co-induced modules.
More precisely, for a partition $\lambda$ of $r$, we have that
\begin{align*}
\Delta(\lambda) &\cong \Ind_{G_{e_r}}(S(\lambda))\\
\nabla(\lambda) &\cong \Coind_{G_{e_r}}(S(\lambda)).
\end{align*}
Note that, if $\lambda$ is a partition of $n$, then we have $\Delta(\lambda)=L(\lambda)=\nabla(\lambda)$.

\subsubsection{(Co)standard versus simple $\Bbbk\mathfrak T_n$-modules}\label{s3.3}

We have the \emph{natural} $\Bbbk \mathfrak T_n$-module $N=\Bbbk^n$ in which
the module structure is defined by putting $fv_i=v_{f(i)}$, for $f\in \mathfrak T_n$, where $v_1,\ldots,v_n$
is the standard basis for $\Bbbk^n$.  Let $\mathrm{Aug}(N)$ be the
\emph{augmentation submodule} of $N$, it consists of all vectors whose coordinates sum to zero.

\begin{theorem}\label{t:Tn.rep}
Let $\lambda$ be a partition of $r$ with $1\leq r\leq n$.
\begin{enumerate}[$($i$)$]
\item\label{t:Tn.rep.1} If $\lambda\neq (1^r)$, then $\Delta(\lambda)=L(\lambda)$.
\item\label{t:Tn.rep.2} If $\lambda=(1^r)$, then $\displaystyle \Delta(\lambda) = P(\lambda) \cong \Lambda^r(N)$ and
$\displaystyle L(\lambda) = \Lambda^{r-1}(\mathrm{Aug}(N))$.
\item\label{t:Tn.rep.3} For each $1\leq r\leq n-1$, there is a short exact sequence
\begin{equation}\label{eq:pascal.identity}
0\longrightarrow L((1^{r+1}))\longrightarrow P((1^r))\longrightarrow L((1^r))\longrightarrow 0.
\end{equation}
\item\label{t:Tn.rep.4} We have $P((1^n))=L((1^n))$ is the one-dimensional sign representation of
$\mathfrak S_n$, extended to $\mathfrak T_n$ by sending all singular mappings to zero.
\end{enumerate}
\end{theorem}
\begin{proof}
Other than the claim $\Delta((1^r))=P((1^r))$, the theorem is just a restatement of Theorem~\ref{t:full.trans.char}, Theorem~\ref{t:ext.is.proj} and Theorem~\ref{t:ext.power.seq}.  Corollary~\ref{c:proj.indec}  implies that $P((1^r))$ has no composition factor $L(\lambda)$ with $\lambda\nleq (1^r)$ and hence $P((1^r))=\Delta((1^r))$ (alternatively, one can easily find a direct isomorphism between $\Ind_{G_{e_r}}(S((1^r)))$ and $\Lambda^r(N)$).
\end{proof}

\subsection{Tilting modules for $\Bbbk\mathfrak T_n$ and the Ringel dual}\label{s4}

\subsubsection{Multiplicities of $L((1^r))$ in injective modules}\label{s4.1}

Our goal is to show that the indecomposable tilting modules for $\Bbbk\mathfrak T_n$ with respect to the
quasi-hereditary structure mentioned above are precisely the injective indecomposable
modules $I(\lambda)$, where  $\lambda\neq (1)$, together with the simple projective module $L((1^n))$.
We begin by studying multiplicities of $L((1^r))$ in injective modules.
We shall write $[V:L]$ for the multiplicity of
a simple module $L$ as a composition factor of a $\Bbbk \mathfrak T_n$-module $V$.

\begin{proposition}\label{p:no.exterior.comp}
Let $1\leq r\leq n$ and let $\lambda$ be a partition of $r$. Then we have
\begin{displaymath}
[I(\lambda):L((1^r))]=
\begin{cases}
1, & \lambda= (1^r);\\
1, & \lambda= (1^{r+1}),\ 1\leq r\leq n-1;\\
0, & \text{else.}
\end{cases}
\end{displaymath}
\end{proposition}
\begin{proof}
As $\Bbbk$ is a splitting field for $\mathfrak T_n$, we have that
\begin{displaymath}
[I(\lambda):L((1^r))]=\dim \mathrm{Hom}(P((1^r)),I(\lambda))=[P((1^r)):L(\lambda)].
\end{displaymath}
Therefore, the claim follows directly from Theorem~\ref{t:Tn.rep}\eqref{t:Tn.rep.3},\eqref{t:Tn.rep.4}.
\end{proof}

\subsubsection{Directed injective tilting modules}\label{s4.2}

From Proposition~\ref{p:no.exterior.comp}, we easily
deduce that a vast majority of the injective indecomposable modules are tilting modules.

\begin{corollary}\label{c:some.tilt}
Let $1\leq r\leq n$ and $\lambda$ be a partition of $r$ different from $(1^r)$.
Then $T(\lambda)=I(\lambda)$.
\end{corollary}

\begin{proof}
Any injective indecomposable module for a quasi-hereditary algebra always has a filtration by costandard modules.  By Proposition~\ref{p:no.exterior.comp}, each composition factor of $I(\lambda)$ is of the form $L(\nu)$ with $\nu$ not of the form $(1^s)$.  But then $\Delta(\nu)=L(\nu)$ by Theorem~\ref{t:Tn.rep} and so a composition series for $I(\lambda)$ is a filtration by standard modules.  Thus $I(\lambda)$ is an indecomposable tilting module with socle $L(\lambda)=\Delta(\lambda)$.  We conclude that $I(\lambda)=T(\lambda)$.
\end{proof}

\subsubsection{Indecomposable tilting $\Bbbk \mathfrak T_n$-modules}\label{s4.3}

We can now describe all indecomposable tilting $\Bbbk \mathfrak T_n$-modules.

\begin{theorem}\label{t:main.tilt}
Let $\Bbbk$ be a field of characteristic $0$ and $\lambda$ a partition of $r$, where $1\leq r\leq n$.
Then we have
\begin{displaymath}
T(\lambda) = \begin{cases}L((1^n)), & \text{if}\ \lambda=(1^n);\\ I(1^{r+1}),
& \text{if}\ \lambda=(1^r),\ 1\leq r\leq n-1;\\ I(\lambda), &\text{else.}\end{cases}
\end{displaymath}
\end{theorem}

\begin{proof}
Since $P((1^n))=L((1^n))=\Delta((1^n))=\nabla((1^n))$, we have $T((1^n))=L((1^n))$. Hence,
by Corollary~\ref{c:some.tilt}, it remains to show that $T((1^r))=I((1^{r+1}))$, for every $1\leq r\leq n-1$.

The module $I((1^{r+1}))$ has simple socle $L((1^{r+1}))$, which is a submodule
of $\Delta((1^{r}))=P((1^r))$ by \eqref{eq:pascal.identity}.  Hence, by injectivity of
$I((1^{r+1}))$, the inclusion $L((1^{r+1}))\hookrightarrow I((1^{r+1}))$
extends to a non-zero homomorphism $\varphi\colon \Delta((1^r))\to I((1^{r+1}))$,
which must be injective because $L((1^{r+1}))=\mathrm{rad}(\Delta((1^r)))$ is the unique maximal submodule of $\Delta((1^r))$ by~\eqref{eq:pascal.identity}.  Since $\Delta((1^r))=P((1^r))$ has exactly two composition factors by Theorem~\ref{t:Tn.rep}\eqref{t:Tn.rep.3}, namely
$L((1^r))$ and $L((1^{r+1}))$, we deduce from Proposition~\ref{p:no.exterior.comp}
that the cokernel $V=I((1^{r+1}))/\Delta((1^r))$
has no composition factor of the form $L((1^s))$, where  $1\leq s\leq n$.
Thus, by Theorem~\ref{t:Tn.rep}, every simple
subquotient of $V$ is a standard module and hence $V$
has a standard filtration.  As $I((1^{r+1}))$ has a costandard filtration (being injective),
we obtain that $I((1^{r+1}))$ is a tilting module. Moreover, from the
definitions we also have that $I((1^{r+1}))\cong T((1^r))$.
\end{proof}

\subsubsection{The Ringel dual}
Let $I_1$ be the ideal of $\mathfrak T_n$ consisting of the constant mappings.  We show that the Ringel dual of $\Bbbk\mathfrak T_n$ with respect to the quasi-hereditary structure we have been considering is Morita equivalent to a one-point extension of $\Bbbk \mathfrak T_n/\Bbbk I_1$.  Here we use that the Ringel dual is Morita equivalent to $\End_{\Bbbk \mathfrak T_n}(T')^{op}$ for any tilting module $T'$ that contains each $T(\lambda)$ as a direct summand.

Let $D$ be the standard duality between right and left $\Bbbk\mathfrak T_n$-modules, so $D(V)=\Hom_{\Bbbk}(V,\Bbbk)$ for a right/left $\Bbbk\mathfrak T_n$-module. Note that $D$ sends projective/injective modules to injective/projective modules and simple modules to simple modules.  Let $P'(\lambda)$ denote the right projective cover of the simple module $L'(\lambda)=D(L(\lambda))$. Then as a right module
\[\Bbbk \mathfrak T_n = \bigoplus_{\lambda\in \Lambda} P'(\lambda)^{\dim L'(\lambda)}\]
and, in particular, $P'((1))$ appears with multiplicity one in this decomposition.  Note that the constant mapping $e_1$ is a primitive idempotent with $P'((1))=e_1\Bbbk \mathfrak T_n=L'((1))$.  To ease notation, we put $e=e_1$.  Then
\begin{equation}\label{eq:right.decomp}
(1-e)\Bbbk T_n=\bigoplus_{\lambda\in \Lambda\setminus\{(1)\}}P'(\lambda)^{\dim L'(\lambda)}
\end{equation}
and if $V=L'((1^n))\oplus (1-e)\Bbbk T_n$, then
\[T'=D(V)=L((1^n))\oplus \bigoplus_{\lambda\in\Lambda\setminus\{(1)\}} I(\lambda)^{\dim L(\lambda)}\]
is a tilting module containing each indecomposable tilting module as a direct summand by Theorem~\ref{t:main.tilt}.
Thus $A=\End_{\Bbbk \mathfrak T_n}(T')^{op}\cong \End_{\Bbbk \mathfrak T_n^{op}}(V)$ is Morita equivalent to the Ringel dual of $\Bbbk \mathfrak T_n$.

Clearly, we have
\begin{align*}
A&\cong \begin{bmatrix} \End_{\Bbbk \mathfrak T_n^{op}}(L'((1^n))) & \Hom_{\Bbbk \mathfrak T_n^{op}}((1-e)\Bbbk\mathfrak T_n,L'((1^n)))\\
                         \Hom_{\Bbbk \mathfrak T_n^{op}}(L'((1^n)),(1-e)\Bbbk\mathfrak T_n) & \End_{\Bbbk\mathfrak T_n^{op}}((1-e)\Bbbk\mathfrak T_n)\end{bmatrix}\\
                         &\cong \begin{bmatrix} \Bbbk & L'((1^n))(1-e)\\ 0 & (1-e)\Bbbk \mathfrak T_n(1-e)\end{bmatrix}
\end{align*}
because $\Hom_{\Bbbk \mathfrak T_n^{op}}(L'((1^n)),(1-e)\Bbbk\mathfrak T_n)=0$ by the dual of Proposition~\ref{p:no.exterior.comp}, as $[(1-e)\Bbbk \mathfrak T_n:L'((1^n))]=1$ and the occurrence is as the simple top of $P'((1^n))$, not in the socle.

Now if $f,g\in \mathfrak T_n$, then $(1-e)f(1-e)=f-fe$ and $(f-fe)(g-ge) = fg-fge$.  Also, note that $f-fe=0$ for any constant mapping $f$.  Thus there is an isomorphism from $\Bbbk \mathfrak T_n/\Bbbk I_1$ to $(1-e)\mathfrak T_n(1-e)$ sending the coset of $f\in \mathfrak T_n\setminus I_1$ to $f-fe$.  The coset of $f$ acts on the right of $L'((1^n))(1-e)$ by multiplication by $\sgn(f)$ where we extend $\sgn$ to $\mathfrak T_n$ by sending non-permutations to $0$.   In conclusion, we have proved the following theorem.

\begin{theorem}\label{t:Ringel.dual}
The Ringel dual of $\Bbbk \mathfrak T_n$ is Morita equivalent to the one-point extension of $\Bbbk \mathfrak T_n/\Bbbk I_1$
\[\begin{bmatrix} \Bbbk & \Bbbk \\ 0&\Bbbk \mathfrak T_n/\Bbbk I_1\end{bmatrix}\]
where $I_1$ is the ideal of constant mappings and where $\Bbbk$ is made a right $\Bbbk T_n/\Bbbk I_1$-module via the extension of the sign representation of $\mathfrak S_n$ to $\mathfrak T_n$ that vanishes on $\mathfrak T_n\setminus \mathfrak S_n$ (and, in particular, on $I_1$).
\end{theorem}

\end{document}